\documentclass{amsart}[12pt,a4paper]

\usepackage{latexsym}
\usepackage{amsfonts,amsmath,amssymb,indentfirst}
\usepackage[english]{babel}
\usepackage[all]{xy}
\usepackage[pdftex]{hyperref}
\usepackage{MnSymbol}
\newcommand{\bdism}{\begin{displaymath}}
\newcommand{\edism}{\end{displaymath}}

\newcommand{\rr}{\mathbb{R}}
\newcommand{\qq}{\mathbb{Q}}

\newcommand{\pp}{\mathbb{P}}

\newcommand{\oo}{\mathcal{O}}

\DeclareMathOperator{\vol}{vol}

\DeclareMathOperator{\red}{red}

\DeclareMathOperator{\ch}{char}

\newtheorem{thm}{Theorem}[section]
\newtheorem{proposition}[thm]{Proposition}
\newtheorem{corollary}[thm]{Corollary}
\newtheorem{lemma}[thm]{Lemma}
\newtheorem{rmk}[thm]{Remark}

\newtheorem{definition}[thm]{Definition}

\newtheorem{question}[thm]{Question}
\newtheorem{conjecture}[thm]{Conjecture}
\newtheorem{construction}[thm]{Construction}

\begin{document}

\title{Effective Matsusaka's Theorem for surfaces in characteristic $p$}

\author{Gabriele Di Cerbo}
\address{Department of Mathematics, Columbia University, New York NY 10027, USA} 
\email{dicerbo@math.columbia.edu}

\author{Andrea Fanelli}
\address{Department of Mathematics, Imperial College London, 180 Queen's Gate,
London SW7 2AZ, UK.}
\email{a.fanelli11@imperial.ac.uk}

\begin{abstract}
We obtain an effective version of Matsusaka's theorem for arbitrary smooth algebraic surfaces in positive characteristic, which provides an effective bound on the multiple which makes an ample line bundle $D$ very ample. The proof for pathological surfaces is based on a Reider-type theorem. As a consequence, a Kawamata-Viehweg-type vanishing theorem is proved for arbitrary smooth algebraic surfaces in positive characteristic.
\end{abstract}

\maketitle

\tableofcontents

\section{Introduction}

A celebrated theorem of Matsusaka (cf. \cite{matsusaka}) states that for a smooth $n$-dimensional complex projective variety $X$ and an ample divisor $D$ on it, there exists a positive integer $M$, only depending on the Hilbert polynomial $\chi(X,\oo_X(kD))$, such that $mD$ is very ample for all $m\ge M$. Koll\'ar and Matsusaka improved the result in \cite{kollar_matsusaka}, showing that the integer $M$ only depends on the intersection numbers $(D^n)$ and $(K_X \cdot D^{n-1})$.

The first effective versions of this result are due to Siu (cf. \cite{siu}, \cite{siu1}) and Demailly (cf. \cite{demailly}, \cite{demailly1}): their methods are cohomological and rely on vanishing theorems. See also \cite{lazarsfeld1} for a full account of this approach.

Although the minimal model program for surfaces in positive characteristic has been recently established, thanks to the work of Tanaka (cf. \cite{tanaka1} and \cite{tanaka2}), some interesting effectivity questions remain open in this setting, after the influential papers of Ekedahl and Shepherd-Barron (cf. \cite{ekedahl} and \cite{shepherdbarron}).

The purpose of this paper is to present a complete solution for the following problem.

\begin{question}\label{qmats}
Let $X$ be a smooth surface over an algebraically closed field of positive characteristic, let $D$ and $B$ be an ample and a nef divisor on $X$ respectively. Then there exists an integer $M$ depending only on $(D^2), \ (K_X\cdot D)$ and $(D\cdot B)$ such that
$$mD-B$$
is very ample for all $m \ge M.$
\end{question}

The analogous question in characteristic zero with $B=0$ was totally solved in
\cite{fernandezdelbusto} and a modified technique allows one to partially extend the result in positive characteristic (cf. \cite{ballico}).



The main result of this paper is the following.


\begin{thm}\label{main}
Let $D$ and $B$ be respectively an ample divisor and a nef divisor on a smooth surface $X$ over an algebraically closed field $k$, with $\ch k =p>0$. Then $mD-B$ is very ample for any
$$m> \frac{2D\cdot (H+B)}{D^{2}}\left((K_{X}+2D)\cdot D+1\right),$$
where
\begin{itemize}
\item $H:=K_{X}+4D$, if $X$ is neither quasi-elliptic with $\kappa(X)=1$ nor of general type;
\item $H:=K_{X}+8D$, if $X$ is quasi-elliptic with $\kappa(X)=1$ and $p=3$;
\item $H:=K_{X}+19D$, if $X$ is quasi-elliptic with $\kappa(X)=1$ and $p=2$;
\item $H:=2K_{X}+4D$, if $X$ is of general type and $p\ge3$;
\item $H:=2K_{X}+19D$, if $X$ is of general type and $p=2$.
\end{itemize}
\end{thm}

The effective bound obtained with $H=K_{X}+4D$  is expected to hold for all surfaces. Note that this bound is not far from being sharp even in characteristic zero (cf. \cite{fernandezdelbusto}).  

The proof of Theorem \ref{main} does not rely directly on vanishing theorems, but rather on Fujita's conjecture on base point freeness and very ampleness of adjoint divisors, which is known to hold for smooth surfaces in characteristic zero (cf. \cite{reider}) and for smooth surfaces in positive characteristic which are neither quasi-elliptic with $\kappa(X)=1$ nor of general type (cf. \cite{shepherdbarron} and \cite{terakawa}).

\begin{conjecture}[Fujita]
Let $X$ be a smooth $n$-dimensional projective variety and let $D$ be an ample divisor on it. Then $K_X+kD$ is base point free for $k\ge n+1$ and very ample for $k\ge n+2$.
\end{conjecture}

If Fujita's conjecture on very ampleness holds then the bound of Theorem \ref{main} with $H=K_{X}+4D$ would work for arbitrary smooth surfaces in positive characteristic. 

For surfaces which are quasi-elliptic with $\kappa(X)=1$ or of general type we can prove the following effective result in the spirit of Fujita's conjecture (cf. Section 4).

\begin{thm}\label{main2}
Let $X$ a smooth surface over an algebraically closed field of characteristic $p>0$, $D$ an ample Cartier divisor on $X$ and let $L(a,b):=aK_X+bD$ for positive integers $a$ and $b$. Then $L(a,b)$ is very ample for the following values of $a$ and $b$:
\begin{enumerate}
\item if $X$ is quasi-elliptic with $\kappa(X)=1$ and $p=3$, $a=1$ and $b\ge 8$;
\item if $X$ is quasi-elliptic with $\kappa(X)=1$ and $p=2$, $a=1$ and $b\ge 19$;
\item if $X$ is of general type with $p\ge3$, $a=2$ and $b\ge 4$;
\item if $X$ is of general type with $p=2$, $a=2$ and $b\ge 19$.
\end{enumerate}
\end{thm}

The key ingredient of Theorem \ref{main2} is a combination of a Reider-type result due to Shepherd-Barron and bend-and-break techniques.

For other results on the geography of pathological surfaces of Kodaira dimension smaller than two, see the recent work \cite{langer}.

In Section 5, a Kawamata-Viehweg-type vanishing theorem which holds for surfaces which are quasi-elliptic with $\kappa(X)=1$ or of general type is proved (cf. Theorem \ref{injectivity} and Corollary \ref{finale}): this generalises the vanishing result in \cite{terakawa}.

The core of our approach is a beautiful construction first introduced by Tango for the case of curves (cf. \cite{tango}) and Ekedahl (cf. \cite{ekedahl}) and Shepherd-Barron (cf. \cite{shepherdbarron}) for surfaces. The same strategy was generalised by Koll\'ar in \cite{kollar} in order to investigate the geography of varieties where Kodaira-type vanishing theorems fail, via the bend-and-break techniques.



\section*{Acknowledgements}

We would like to thank Paolo Cascini, Yoshinori Gongyo and Hiromu Tanaka for useful comments and the referee for helping us to improve our exposition. 

The second author is funded by the grant {\it Roth Scholarship} of the Department of Mathematics at Imperial College London.

\section{Preliminary results}

In this section we recall some techniques we will need later in this paper.

\subsection{Volume of divisors}

Let $D$ be a Cartier divisor on a normal variety $X$, not necessarily a surface. The volume of $D$ measures the asymptotic growth of the space of global sections of multiples of $D$. We will recall here few properties of the volume and we refer to \cite{lazarsfeld1} for more details. 

\begin{definition}
Let $D$ be a Cartier divisor on $X$ with $\dim(X)=n$. The volume of $D$ is defined by
\bdism
\vol(D):=\limsup_{m\to \infty}\frac{h^{0}(X,\oo_{X}(mD))}{m^{n}/n!}.
\edism
The volume of $X$ is defined as $\vol(X):=\vol(K_X)$.
\end{definition}

It is easy to show that if $D$ is big and nef then $\vol(D)=D^{n}$. In general, it is a hard invariant to compute but thanks to Fujita's approximation theorem, some of its properties can be deduced from the case where $D$ is ample. For a proof of the theorem in characteristic $0$ we refer to \cite{lazarsfeld1}. More recently, Takagi gave a proof of the same theorem in positive characteristic (see \cite{takagi}). In particular, we can deduce the log-concavity of the volume function even in positive characteristic. The proof is exactly the same as Theorem 11.4.9 in \cite{lazarsfeld1}.

\begin{thm} \label{logconc}
Let $D$ and $D'$ be big Cartier divisors on a normal variety $X$ defined over an algebraically closed field. Then 
\bdism
\vol(D+D')^{1/n}\geq \vol(D)^{1/n}+\vol(D')^{1/n}.
\edism
\end{thm}

\subsection{Bogomolov's inequality and Sakai's theorems}

We start with the notion of semi-stability for rank-two vector bundles on surfaces. Let $X$ be a smooth surface defined over an algebraically closed field.

\begin{definition}\label{inst_bog}
A rank-two vector bundle $\mathcal{E}$ on $X$ is {\it unstable} if it fits in a short exact sequence
$$0 \rightarrow \oo_X(D_1)\rightarrow \mathcal{E}\rightarrow \mathcal{I}_Z \cdot \oo_X(D_2)\rightarrow 0$$
where $D_1$ and $D_2$ are Cartier divisors such that $D':=D_1-D_2$ is big with $(D'^2)>0$ and $Z$ is an effective $0$-cycle on $X$.

The vector bundle $\mathcal{E}$ is {\it semi-stable} if it is not unstable.
\end{definition}

In characteristic zero, the following celebrated result holds (cf. \cite{bogomolov}).

\begin{thm}[Bogomolov]
Let $X$ be defined over a field of characteristic zero. Then every rank-two vector bundle $\mathcal{E}$ for which $c_1^2(\mathcal{E}) > 4c_2(\mathcal{E})$ is unstable.
\end{thm}

As a consequence, one can deduce the following theorem, due to Sakai (cf. \cite[Proposition 1]{sakai}), which turns out to be an equivalent statement 
as shown in \cite{dicerbo}.

\begin{thm}\label{sakai_1}
Let $D$ be a nonzero big divisor with $D^2 > 0$ on a smooth projective surface $X$ over a field of characteristic zero. If $H^1(X, \oo_X (K_X + D))\neq 0$ then there exists a non-zero effective divisor $E$ such that
\begin{itemize}
\item $D - 2E$ is big; 
\item $(D-E)\cdot E \le 0$.
\end{itemize}
\end{thm}

The previous result easily implies a weaker version of Reider's theorem.

\begin{thm}\label{sakai_2}
Let $D$ be a nef divisor with $D^2 > 4$ on a smooth projective surface $X$ over a field of characteristic zero. Then $K_X + D$ has no base point unless there exists a non-zero effective divisor $E$ such that $D \cdot E = 0$ and $(E^2) = -1$ or $D \cdot E = 1$ and $(E^2) = 0$.
\end{thm}

The following result, conjectured by Fujita, can be deduced for smooth surfaces in characteristic zero.

\begin{corollary}[Fujita Conjectures for surfaces, char 0]\label{shepherd}
Let $D_1, \ldots, D_k$ be ample divisors on $X$ smooth, over a field of characteristic zero. Then $K_X+D_1+\ldots+D_k$ is base-point free if $k\ge3$ and very ample if $k\ge 4$.
\end{corollary}

We remark that Theorem \ref{sakai_1} is not known in general for smooth surfaces in positive characteristic, although Fujita's conjectures are expected to hold.

\subsection{Ekedahl's construction and Shepherd-Barron's theorem}

In this section we recall some classical results on the geography of smooth surfaces in positive characteristic (see \cite{ekedahl}, \cite{shepherdbarron} and \cite{shepherdbarron1}).

For a good overview on the geography of surfaces in positive characteristic, see \cite{liedtke}.


We discuss here a construction which is due to Tango for the case of curves (cf. \cite{tango}) and Ekedahl (cf. \cite{ekedahl}) for surfaces. There are many variations on the same theme, but we will focus on the one which is more related to stability of vector bundles. We need this fundamental result.

\begin{thm}[Bogomolov]\label{pe_unstable}
Let $\mathcal{E}$ be a rank-two vector bundle on a smooth projective surface $X$ over a field of positive characteristic such that the Bogomolov's inequality does not hold (i.e. such that $c_1^2(\mathcal{E}) > 4c_2(\mathcal{E})$). Then there exists a reduced and irreducible surface $Y$ contained in the ruled threefold $\pp(\mathcal{E})$ such that
\begin{itemize}
\item the restriction $\rho\colon Y\rightarrow X$ is $p^e$-purely inseparable for some $e>0$;
\item $(F^*)^e (\mathcal{E})$ is unstable.
\end{itemize}
\end{thm}

\proof
See \cite[Theorem 1]{shepherdbarron}
\endproof

The previous result also provides an explicit construction of the purely inseparable cover (cf. \cite{shepherdbarron}).

\begin{construction}\label{constr}
Take a rank-two vector bundle $\mathcal{E}$ such that the Bogomolov's inequality does not hold and let $e$ be an integer such that $F^{e*}\mathcal{E}$ is unstable. We have the following commutative diagram
$$\begin{xymatrix}
{
  \pp(F^{e*}\mathcal{E}) \ar[r]^G \ar[d]_{p'} &  \pp(\mathcal{E})\ar[d]^p\\
  X  \ar[r]^{F^e}   &  X
 }
 \end{xymatrix}.$$
The fact that $F^{e*}\mathcal{E}$ is unstable gives an exact sequence
$$0 \rightarrow \oo_X(D_1)\rightarrow F^{e*}\mathcal{E}\rightarrow \mathcal{I}_Z \cdot \oo_X(D_2)\rightarrow 0$$
and a quasi-section $X_0$ of $\pp(F^{e*}\mathcal{E})$ (i.e. $p'_{|X_0}\colon X_0\rightarrow X$ is birational). Let $Y$ be the image of $X_0$ via $G$. One can show that the induced morphism.
$$\rho\colon Y \rightarrow X$$
is $p^e$-purely inseparable. Let us define $D':=D_1-D_2$. One can show (cf. \cite[Corollary 5]{shepherdbarron}) that
$$K_Y \equiv \rho^*\bigg(K_X - \frac{p^e-1}{p^e}D'\bigg).$$
\end{construction}

\begin{rmk}\label{rem_eff}
We will be particularly interested in the case when the rank-two vector bundle $\mathcal{E}$ comes as a nontrivial extension
$$0\rightarrow \oo_X \rightarrow \mathcal{E} \rightarrow \oo_X(D)\rightarrow 0$$
associated to a non-zero element $\gamma \in H^1(X,\oo_X(-D)),$ where $D$ is a big Cartier divisor such that $(D^2)>0$. Indeed, the instability of $F^{e*}\mathcal{E}$ guarantees the existence of a diagram (keeping the notation as in Definition \ref{inst_bog})
$$\begin{xymatrix}
{
  & & 0 \ar[d] & & \\
  & & \oo_X(D'+D_2)\ar[d]^{f_1} & & \\
  0 \ar[r] & \oo_X \ar[r]^{g_1} \ar[rd]_\sigma & F^{e*}\mathcal{E} \ar[r]^{g_2} \ar[d]^{f_2} & \oo_X(p^e D) \ar[r] & 0 \\
  & & \mathcal{I}_Z \cdot \oo_X(D_2) \ar[d] & & \\
  & & 0 & & 
 }
 \end{xymatrix}.$$
First, we claim that the composition map $\sigma$ is non-zero. Assume, by contradiction that $\sigma \equiv 0$. This gives a nonzero section $\sigma'\colon \oo_X \rightarrow \oo_X(D'+D_2)$. This forces the composition $\tau:= g_2 \circ f_1$ to be zero. But this implies that $D'+D_2 \le 0$. This is a contradiction (cf. the proof of \cite[Proposition 1]{sakai} and \cite[Lemma 16]{shepherdbarron}). 

This implies that $D_2 \simeq E \ge 0$: one can then rewrite the vertical exact sequence as follows:
$$0 \rightarrow \oo_X(p^eD-E)\rightarrow F^{e*}\mathcal{E}\rightarrow \mathcal{I}_Z \cdot \oo_X(E)\rightarrow 0.$$

\end{rmk}

Since \cite[Corollary 8]{shepherdbarron} guarantees that Corollary \ref{shepherd} holds true for smooth surfaces in positive characteristic which are neither quasi-elliptic with $\kappa(X)=1$ nor of general type, we need to deduce effective base-point-freeness and very ampleness results only for these two classes of surfaces.

We recall here the following key result from \cite{shepherdbarron}.

\begin{thm}
Let $\mathcal{E}$ be a rank-two vector bundle on a smooth projective surface $X$ over an algebraically closed field of positive characteristic such that the Bogomolov's inequality does not hold and is semi-stable. Then
\begin{itemize}
\item if $X$ is not of general type, then $X$ is quasi-elliptic with $\kappa(X)=1$;
\item if X is of general type and 
$$c_1^2(\mathcal{E}) - 4c_2(\mathcal{E}) > \frac{\vol(X)}{(p-1)^2},$$
then $X$ is purely inseparably uniruled. More precisely, in the notation of Theorem \ref{pe_unstable}, $Y$ is uniruled.
\end{itemize}
\end{thm}

\proof
This is \cite[Theorem 7]{shepherdbarron}, since the volume of a surface $X$ with minimal model $X'$ equals $(K_{X'}^2).$
\endproof

The consequence we are interested in is the following (cf. \cite[Corollary 8]{shepherdbarron}).

\begin{corollary}[Shepherd-Barron]\label{unir}
Corollary \ref{shepherd} holds in positive characteristic if $X$ is neither of general type nor quasi-elliptic.
\end{corollary}

\subsection{Bend-and-break lemmas}

We recall here a well-known result in birational geometry, based on a celebrated method due to Mori (see \cite{kollar} for an insight into these techniques).

First we need to recall some notation. Mori theory deals with effective 1-cycles in a variety $X$; more specifically we will consider non-constant morphisms $h\colon C\rightarrow X$, where $C$ is a smooth curve. In particular, these techniques allow to deform curves for which
$$(K_X \cdot C):= \deg_C h^* K_X <0.$$
In what follows, we will denote with $\overset{e}{\approx}$ the {\it effective algebraic equivalence} defined on the space of effective 1-cycles $Z_1(X)$ (see \cite[Definition II.4.1]{kollar}).

\begin{thm}[Bend-and-break]\label{b-a-b}
Let $X$ be a variety over an algebraically closed field and let $C$ be a smooth, projective and irreducible curve with a morphism $h\colon C\rightarrow X$ such that $X$ has local complete intersection singularities along $h(C)$ and $h(C)$ intersects the smooth locus of $X$. Assume the following numerical condition holds:
$$(K_X\cdot C) <0.$$
Then for every point $x \in C$, there exists a rational curve $C_x$ in $X$ passing through $x$ such that 
\begin{equation}\label{alg_rel1}
h_*[C]\overset{e}{\approx} k_0 [C_x] + \sum_{i\neq 0} k_i[C_i]
\end{equation} 
(as algebraic cycles) with $k_i \ge 0$ for all $i$ and
$$-(K_X\cdot C_x) \le \dim X +1.$$
\end{thm}

\proof
See \cite[Theorem II.5.14 and Remark II.5.15]{kollar}. The relation (\ref{alg_rel1}) can be deduced looking directly at the proofs of the bend-and-break lemmas (cf. \cite[Corollary II.5.6 and Theorem II.5.7]{kollar}): our notation is slightly different, since in (\ref{alg_rel1}) we have isolated a rational curve with the required intersection properties.
\endproof

In this paper we will need the following consequence of the previous theorem.

\begin{corollary}\label{b-a-b-cor}
Let $X$ be a surface which fibres over a curve $C$ via $f\colon X\rightarrow C$ and let $F$ be the general fibre of $f$. Assume that $X$ has only local complete intersection singularities along $F$ and that $F$ is a (possibly singular) rational curve such that
$$(K_X\cdot F) <0.$$
Then
$$-(K_X\cdot F) \le 3.$$
\end{corollary}

\proof
We are in the hypothesis of Theorem \ref{b-a-b}, so we can take a point $x$ in the smooth locus of $X$ and deduce the existence of a rational curve $C'$ passing through $x$ such that
$$-(K_X\cdot C') \le 3 \mbox{   and } [F]\overset{e}{\approx} k_0[C'] +\sum_{i\neq 0} k_i[C_i].$$

By Exercise II.4.1.10 in \cite{kollar}, the curves appearing on the right hand side of the previous equation must be contained in the fibers of $f$.  
Since $F$ is the general fibre, the second relation implies that $k_0=1$ and $k_i=0$ for all $i\neq 0$ and so that $C'=F$.
\endproof




\section{Effective Matsusaka's theorem}\label{6}

In this section we prove Theorem \ref{main} assuming the results on effective very ampleness that we will prove in the next section. If not specified, $X$ will denote a smooth surface over an algebraically closed field of arbitrary characteristic.

First, let us recall the following numerical criterion for bigness, whose characteristic-free proof is based on Riemann-Roch (cf. \cite{lazarsfeld}, Theorem 2.2.15).

\begin{thm}\label{bigcr}
Let $D$ and $E$ be nef $\qq$-divisors on $X$ and assume that
$$D^2>2(D\cdot E).$$
Then $D-E$ is big.
\end{thm}

Before proving Theorem \ref{main}, we need some lemmas.

\begin{lemma} \label{trick}
Let $D$ be an ample divisor on $X$. Then $K_{X}+2D+C$ is nef for any irreducible curve $C\subset X$. 
\end{lemma}

\begin{proof}
If $X=\pp^{2}$ then the lemma is trivial. By the cone theorem and the classification of surfaces with extremal rays of maximal length, we have that $K_{X}+2D$ is always a nef divisor. This implies that $K_{X}+2D+C$ may have negative intersection number only when intersected with $C$. On the other hand, by adjunction, $(K_{X}+C)\cdot C=2g-2\geq -2$, where $g$ is the arithmetic genus of $C$. Since $D$ is ample, the result follows.
\end{proof}

We can now prove on of the main results of this section (cf. \cite{lazarsfeld1}, Theorem 10.2.4).

\begin{thm}\label{matsu}
Let $D$ be an ample divisor and let $B$ be a nef divisor on $X$. Then $mD-B$ is nef for any 
$$
m\geq \frac{2D\cdot B}{D^{2}}\left((K_{X}+2D)\cdot D+1\right)+1.
$$
\end{thm}

\begin{proof}
To simplify the notation in the proof let us define the following numbers: 
\begin{align}\notag
&\eta=\eta(D,B):=\inf\left\{t\in\rr_{>0}\:|\: t D- B \;\text{is nef}\right\}, \\ \notag
&\gamma=\gamma(D,B):=\inf\left\{t\in\rr_{>0}\:|\: t D- B \;\text{is pseudo-effective}\right\}
\end{align}

The theorem will follow if we find an upper bound on $\eta$. Note that $\gamma\leq \eta$ since a nef divisor is also pseudo-effective.

By definition $\eta D -B$ is in the boundary of the nef cone and by Nakai's theorem we have two possible cases: either
\begin{itemize}
\item $(\eta D-B)^{2}=0$, or 
\item there exists an irreducible curve $C$ such that $(\eta D-B)\cdot C=0$. 
\end{itemize}
If $(\eta D-B)^{2}=0$, then it is easy to see that 
$$
\eta\leq 2\frac{D\cdot B}{D^{2}}.
$$
So we can assume that there exists an irreducible curve $C$ such that $\eta D\cdot C=B\cdot C$. Let us define $G:=\gamma D-B$. Then
$$
G\cdot C=(\gamma-\eta)D\cdot C\leq (\gamma-\eta).
$$
Let us define $A:=K_{X}+2D$. By Lemma \ref{trick} and the definition of $G$, we have that $(A+C)\cdot G\geq 0$. Combining with the previous inequality we get 
$$
(\eta-\gamma)\leq -G\cdot C\leq A\cdot G= \gamma A\cdot D-A\cdot B.
$$
In particular, 
$$\eta\leq \gamma (A\cdot D+1)-A\cdot B\leq \gamma (A\cdot D+1).$$ 
The statement of our result follows from Theorem \ref{bigcr}, which guarantees that $\gamma< \frac{2D\cdot B}{D^{2}}$.

\end{proof}

\begin{rmk}
The previous proof is characteristic-free, although the new result is for surfaces in positive characteristic.
\end{rmk}

We can now prove our main theorem, assuming the results in the next section.

\begin{proof}[Proof of Theorem \ref{main}]
By Corollary \ref{unir}, if $X$ is neither of general type nor quasi-elliptic and $H=K_X+4D$ then $H+N$ is very ample for any nef divisor $N$. By Theorem \ref{matsu}, $mD-(H+B)$ is nef for any $m$ as in the statement. Then $K_X+4D+(mD-K_X-4D-B)$ is very ample. For surfaces in the other classes use Theorem \ref{main_va_qell} and Theorem \ref{mainc} to obtain the desired very ample divisor $H$.
\end{proof}

\section{Effective very ampleness in positive characteristic}\label{6}

The aim of this section is to complete the proof of Theorem \ref{main} for quasi-elliptic surfaces of Kodaira dimension one and for surfaces of general type. Our ultimate goal is to prove Theorem \ref{main2} via a case-by-case analysis.

First we need some notation (cf. Theorem \ref{sakai_1}).

\begin{definition}
A big divisor $D$ on a smooth surface $X$ with $(D^2)>0$ is $m$-unstable for a positive integer $m$ if either 
\begin{itemize}
\item $H^1(X,\oo_X(-D))= 0$; or  
\item $H^1(X,\oo_X(-D))\neq 0$ and  there exists a nonzero effective divisor $E$ such that
\begin{itemize}
\item $mD - 2E$ is big; 
\item $(mD-E)\cdot E \le 0$.
\end{itemize}
\end{itemize}
\end{definition}

\begin{rmk}
Theorem \ref{sakai_1} tells us that in characteristic zero every big divisor $D$ on a smooth surface $X$ with $(D^2)>0$ is $1$-unstable. The same holds in positive characteristic, if we assume that the surface is neither of general type nor quasi-elliptic of maximal Kodaira dimension: this is a consequence of Corollary \ref{unir}. Our goal here is to clarify the picture in the remaining cases.
\end{rmk}

We can start our analysis with quasi-elliptic surfaces of maximal Kodaira dimension.

\begin{proposition}\label{q_ell}
Let $X$ be a quasi-elliptic with $\kappa(X)=1$ and let $D$ be a big divisor on $X$ with $(D^2)>0$. Then
\begin{enumerate}
\item if $p=3$, then $D$ is $3$-unstable;
\item if $p=2$, then $D$ is $4$-unstable.
\end{enumerate}
\end{proposition}

\proof
Assume that $p=3$ and $H^1(X,\oo_X(-D))\neq 0$. This nonzero element gives a non-split extension
$$0\rightarrow \oo_X \rightarrow \mathcal{E} \rightarrow \oo_X(D)\rightarrow 0.$$
Theorem \ref{pe_unstable} implies that $(F^*)^e\mathcal{E}$ is unstable for $e$ sufficiently large.
To prove the proposition in this case we need to show that $e=1$. Assume $e\ge 2$ and let $F$ be the general element of the pencil which gives the fibration in cuspidal curves $f\colon X\rightarrow B$. Let $\rho\colon Y \rightarrow X$ be the $p^e$-purely inseparable morphism of Construction \ref{constr}. Then $\{C_i:=\rho^*F\}$ is a family of movable rational curves in $Y$. Let us define $g:= f \circ \rho$ and consider its Stein factorisation:
$$\xymatrix{
Y \ar[dr]^h \ar[d]_\rho \ar@/_{2.0pc}/[dd]_g \\
X  \ar[d]_f & B' \ar[dl]^c \\
B \\}.$$
Since the family $\{C_i\}$ are precisely the fibres of $h$, we can use Corollary \ref{b-a-b-cor} on $h\colon Y\rightarrow B'$ ($Y$ is defined via a quasi-section in a $\pp^1$-bundle over $X$, so it has hyper surface singularities along the general element of $\{C_i\}$) and deduce that
$$0<-(K_Y\cdot C_i)\le 3.$$
This gives a contradiction, since
$$3\ge-(K_Y\cdot C_i) = \bigg(\rho^*\bigg(\frac{p^e-1}{p^e}(p^eD-2E)-K_X\bigg)\cdot C_i\bigg) = p^e\bigg(\bigg(\frac{p^e-1}{p^e}(p^eD-2E)-K_X\bigg)\cdot  F\bigg)$$
$$=((p^e-1)(p^eD-2E)\cdot F)\ge p^e-1 \ge 8,$$
where $E$ is the one appearing in Remark \ref{rem_eff}.

The same proof works for $p=2$, although in this case we can only prove that $e\le 2.$
\endproof

We can now focus on the general type case. We need the following theorem by Shepherd-Barron (cf. \cite[Theorem 12]{shepherdbarron}).

\begin{thm}\label{eff_reid}
Let $D$ be a big Cartier divisor on a smooth surface $X$ of general type which verifies one of the following hypothesis:
\begin{itemize}
\item $p\ge 3$ and $(D^2)>\vol(X);$
\item $p=2$ and $(D^2)>\max\{\vol(X),\vol(X)-3\chi(\oo_X)+2\}.$
\end{itemize}
Then $D$ is $1$-unstable.
\end{thm}

Since the bound of the previous theorem depends on $\chi(\oo_X)$ if $p=2$, we need an additional result for this case. First we recall a result by Shepherd-Barron (cf. \cite[Theorem 8]{shepherdbarron1}).

\begin{thm}\label{sb_char_2}
Let $X$ be a surface in characteristic $p=2$ of general type with $\chi(\oo_X)<0$. Then there is
a fibration $f\colon X\rightarrow C$ over a smooth curve $C$, whose generic fibre is a singular rational curve with arithmetic genus $2\le g\le 4$.
\end{thm}

We can prove now our result.

\begin{proposition}\label{char_2}
Let $D$ be a big Cartier divisor on a surface in characteristic $p=2$ of general type with $\chi(\oo_X)<0$ such that $(D^2)>\vol(X)$. Then $D$ is $4$-unstable.
\end{proposition}

\proof
Assume that $H^1(X,\oo_X(-D))\neq 0$. As in the proof of Proposition \ref{q_ell} we have a non-split extension
$$0\rightarrow \oo_X \rightarrow \mathcal{E} \rightarrow \oo_X(D)\rightarrow 0.$$
Using Theorem \ref{pe_unstable} we deduce the instability of $(F^*)^e\mathcal{E}$ for $e$ sufficiently large.
Let $F$ be the general element of the pencil which gives the fibration in singular rational curves given by Theorem \ref{sb_char_2}. Let $\rho\colon Y \rightarrow X$ be the $p^e$-purely inseparable morphism of Construction \ref{constr}. Like in the proof of Proposition \ref{q_ell}, we use Corollary \ref{b-a-b-cor} on $Y$ and deduce that $0<-(K_Y\cdot C_i)\le 3.$
This gives
$$3\ge-(K_Y\cdot C_i) = \bigg(\rho^*\bigg(\frac{2^e-1}{2^e}(2^eD-2E)-K_X\bigg)\cdot C_i\bigg) = 2^e\bigg(\bigg(\frac{2^e-1}{2^e}(2^eD-2E)-K_X\bigg)\cdot  F\bigg)$$
$$=(((2^e-1)(2^eD-2E)-2^eK_X)\cdot F)\ge 1.$$
This implies that
$$(((2^e-1)(2^{e-1}D-E)-2^{e-1}K_X)\cdot F)= 1.$$
As a consequence, we apply Theorem \ref{sb_char_2} to bound the intersection $(K_X\cdot F)$:
$$(2^e-1)((2^{e-1}D-E)\cdot F) = 2^{e}(g-1)+1,$$
where $g$ is the arithmetic genus of $F$.
Some basic arithmetic give that the only possibilities for the pair $(g,e)$ are $(2,1)$, $(3,1)$, $(3,2)$ and $(4,1)$.

\endproof


We will use Theorem \ref{eff_reid} to prove a variant of Reider's theorem in positive characteristic. We state a technical proposition 
we will need later (cf. \cite[Proposition 2]{sakai}). 

\begin{proposition}
Let $\pi:Y\rightarrow X$ be a birational morphism between two normal surfaces. Let $\tilde{D}$ be a Cartier divisor on $Y$ such that 
$\tilde{D}^{2}>0$. Assume there is a nonzero effective divisor $\tilde{E}$ such that 
\begin{itemize}
\item $\tilde{D}-2\tilde{E}$ is big and 
\item $(\tilde{D}-\tilde{E})\cdot \tilde{E}\leq 0$.
\end{itemize}
Set $D:=\pi_{*}\tilde D$, $E:=\pi_{*}\tilde{E}$ and $\alpha= D^{2}-\tilde{D}^{2}$. If $D$ is nef and $E$ is a non-zero effective divisor, then
\begin{itemize}
\item $0\leq D\cdot E < \alpha/2$, 
\item $D\cdot E -\alpha/4 \leq E^{2} \leq (D\cdot E)^{2} / D^{2}$.
\end{itemize} 
\end{proposition}

The corollary we need is the following.

\begin{corollary}\label{push}
Let $\pi:Y\rightarrow X$ be a birational morphism between two smooth surfaces and let $\tilde{D}$ be a big Cartier divisor on $Y$ such that $(\tilde{D}^2)>0$. Assume that
\begin{itemize}
\item $H^1(X,\oo_X(-\tilde{D}))\neq 0$;
\item $\tilde{D}$ is $m$-unstable for some $m>0$.
\end{itemize}
Set $D:=\pi_{*}\tilde D$ and $\alpha= D^{2}-\tilde{D}^{2}$. Then if $D$ is nef, there exists a nonzero effective divisor $E$ on $X$ such that
\begin{itemize}
\item $0\leq D\cdot E < m\alpha/2$, 
\item $mD\cdot E -m^2\alpha/4 \leq E^{2} \leq (D\cdot E)^{2} / D^{2}$.
\end{itemize} 

\end{corollary}

We can derive our effective base point freeness results. We will start with quasi-elliptic surfaces, applying Theorem \ref{q_ell} and the previous corollary.

\begin{proposition}\label{quasi_ell_reider}
Let $X$ be a quasi-elliptic surface with maximal Kodaira dimension. Let $D$ be a big and nef divisor on $X$. Then the following holds.
\begin{itemize}
\item $p=3:$ 
\begin{itemize}
\item if $D^{2}>4$ and $|K_{X}+D|$ has a base point at $x\in X$, there exists a curve $C$ such that $D\cdot C \leq 5$;
\item if $D^{2}>9$ and $|K_{X}+D|$ does not separate any two points $x,y\in X$, there exists a curve $C$ such that $D\cdot C \leq 13$;
\end{itemize}
\item $p=2:$ 
\begin{itemize}
\item if $D^{2}>4$ and $|K_{X}+D|$ has a base point at $x\in X$, there exists a curve $C$ such that $D\cdot C \leq 7$;
\item if $D^{2}>9$ and $|K_{X}+D|$ does not separate any two points $x,y\in X$, there exists a curve $C$ such that $D\cdot C \leq 17$;
\end{itemize}
\end{itemize}
\end{proposition} 

\proof
We start with the case $p=3$. Assume that $|K_{X}+D|$ has a base point at $x\in X$. Let $\pi: Y\rightarrow X$ be the blow up at $x$. Since $x$ is a base point we have that 
$H^{1}(Y,\oo_{Y}(K_{Y}+\pi^{*}D-2F))\neq 0$, where $F$ is the exceptional divisor of $\pi$. Let $\tilde{D}:= \pi^{*}D-2F$. By assumption we have that $\tilde{D}^{2}>0$. By Theorem \ref{q_ell} we can find an effective divisor $\tilde{E}$ such that $p\tilde{D}-2\tilde{E}$ is big and $(p\tilde{D}-\tilde{E})\cdot \tilde{E}\leq 0$. The previous inequality easily implies that $\tilde{E}$ is not a positive multiple of the exceptional divisor and in particular $E:=\pi_{*}\tilde{E}$ is a non-zero effective divisor. Moreover, $D=\pi_{*}\tilde{D}$ is nef by assumption, thus we can apply Corollary \ref{push}. Since $\alpha=(D^{2}-\tilde{D}^{2})=4$, the first inequality of the corollary implies that $D\cdot E\leq 5$. 
\\The statement on separation of points follows exactly in the same way. Note that we allow the case $x=y$.

The bounds for the case $p=2$ can be obtained the same way, remarking that $\tilde{D}$ is $p^2$-unstable in this case.
\endproof

The previous results can be used to derive effective very ampleness statements for quasi elliptic surfaces when $D$ is an ample divisor.

\begin{thm}\label{main_va_qell}
Let $D$ be an ample Cartier divisor on a smooth quasi-elliptic surface $X$ with $\kappa(X)=1$. Then 
\begin{itemize}
\item if $p=3$, the divisor $K_X+kD$ is base-point free for any $k\geq 4$ and it is very ample for any $k\geq 8$;
\item if $p=2$, the divisor $K_X+kD$ is base-point free for any $k\geq 5$ and it is very ample for any $k\geq 19$.
\end{itemize}

In particular, if $N$ is any nef divisor, $K_{X}+k D+N$ is always very ample for any $k\geq 8$ (resp. $k\geq 19$) in characteristic $3$ (resp. $2$).
\end{thm}

\proof
The proof consists in explicitly computing the minimal multiple of $D$ which contradicts the second inequality of Corollary \ref{push}.

Let us start with base-point-freeness for $p=3$. Assume that $k\ge 5$, $K_X+kD$ has a base point and define $D':=kD$. Then, by Proposition \ref{quasi_ell_reider}, we know that there exists an effective divisor $E$ such that $(D'\cdot E)\le 5$. This implies
$$(D\cdot E)\le 1.$$
Now use the second inequality of Corollary \ref{push} on $D'$ to deduce
$$ 15 - 9\le 3(D'\cdot E)-9 \le \frac{(D'\cdot E)^2}{(D'^2)}\le 1.$$
This is a contradiction.

Similar computations give the other bounds.
\endproof

We now deal with surfaces of general type. The analogous of Proposition \ref{quasi_ell_reider} is the following.

\begin{proposition}\label{gen_type_reider}
Let $X$ be a surface of general type and let $D$ be a big and nef divisor on $X$. Then the following holds.
\begin{itemize}
\item $p\ge3:$ 
\begin{itemize}
\item if $D^{2}>\vol(X)+4$ and $|K_{X}+D|$ has a base point at $x\in X$, there exists a curve $C$ such that $D\cdot C \leq 1$;
\item if $D^{2}>\vol(X)+9$ and $|K_{X}+D|$ does not separate any two points $x,y\in X$, there exists a curve $C$ such that $D\cdot C \leq 2$;
\end{itemize}
\item $p=2:$ 
\begin{itemize}
\item if $D^{2}>\vol(X)+6$ and $|K_{X}+D|$ has a base point at $x\in X$, there exists a curve $C$ such that $D\cdot C \leq 7$;
\item if $D^{2}>\vol(X)+11$ and $|K_{X}+D|$ does not separate any two points $x,y\in X$, there exists a curve $C$ such that $D\cdot C \leq 17$;
\end{itemize}
\end{itemize}
\end{proposition}

\proof
The proof is basically the same as Proposition \ref{quasi_ell_reider}.
Let $p\ge 3$ and assume that $|K_{X}+D|$ has a base point at $x\in X$. Using the same notation as Proposition \ref{quasi_ell_reider}, we can blow up $x$ and deduce the existence of an effective divisor $\tilde{E}$ such that $\tilde{D}-2\tilde{E}$ is big and $(\tilde{D}-\tilde{E})\cdot \tilde{E}\leq 0$ (in order to deduce $1$-instability we use Theorem \ref{eff_reid}). Also here, the first inequality of Corollary \ref{push} implies that $(D\cdot E)\leq 1$. 
\\The statement on separation of points follow the same way.

The bounds for the case $p=2$ we use the same strategy, using a combination of Theorem \ref{sb_char_2} and Proposition \ref{char_2}.
\endproof

The following effective very ampleness statement can be deduced. Applying directly Proposition \ref{gen_type_reider} would provide bounds that depend on the volume. It is possible to get a uniform bound if we work with linear systems of the type $|2K_{X}+mD|$. Note that we get sharp statements for those linear systems. 

\begin{thm}\label{mainc}
Let $D$ be an ample Cartier divisor on a smooth surface $X$ of general type. Then 
\begin{itemize}
\item if $p\ge 3$, the divisor $2K_X+kD$ is base-point free for any $k\geq 3$ and it is very ample for any $k\geq 4$.
\item if $p=2$ the divisor $2K_X+kD$ is base-point free for any $k\geq 5$ and it is very ample for any $k\geq 19$.
\end{itemize}

In particular, if $N$ is any nef divisor, $2K_{X}+k D+N$ is always very ample for any $k\geq 4$ (resp. $k\geq 19$) in characteristic $p\ge3$ (resp. $p=2$).
\end{thm}

\proof
Since negative extremal rays of general type surfaces have length $1$, if $m\geq 3$, we know that $L:=K_{X}+mD$ is an ample divisor and $L\cdot C\geq 2$ for any irreducible curve $C\subset X$. Moreover, 
by log-concavity of the volume function (see Theorem \ref{logconc}) we have that
\begin{equation}\notag
L^{2}=\vol(L)\geq \vol(K_{X})+9 D^{2}> \vol(X)+4.
\end{equation} 

Proposition \ref{gen_type_reider} implies that $K_{X}+L=2K_{X}+kD$ is base point free for any $k\geq 4$. A similar computation allows us to derive very ampleness. 

The same strategy gives the result for $p=2$.
\endproof

\proof[Proof of Theorem \ref{main2}]
This is simply given by Theorem \ref{main_va_qell} and Theorem \ref{mainc}
\endproof

\begin{rmk}
In \cite{terakawa}, similar results can be found. Nonetheless our approach allows to deduce effective base point free and very ampleness also on quasi-elliptic surfaces and arbitrary surfaces of general type.
\end{rmk}

\section{A Kawamata-Viehweg-type Vanishing Theorem in positive characteristic}

In this section we give an extension of the results in \cite{terakawa}. In that paper, the author uses the results in \cite{shepherdbarron} to deduce a Kawamata-Viehweg-type theorem for non-pathological surfaces. Using our methods we are able to discuss pathological surfaces and obtain an effective Kawamata-Viehweg type theorem in positive characteristic.

Let us first recall the classical Kawamata-Viehweg vanishing theorem in its general version (see \cite{kollarmori} for the general notation).

\begin{thm}\label{kv}
Let $(X,B)$ be a klt pair over an algebraically closed field of characteristic zero and let $D$ be a Cartier divisor on $X$ such that $D-(K_X + B)$ is big and nef. Then 
$$H^i(X, \oo_X(D)) = 0$$
for any $i > 0$.
\end{thm}

In positive characteristic, even for non-pathological smooth surfaces, there are counterexamples to Theorem \ref{kv}: Xie provided examples in \cite{xie1} of relatively minimal irregular ruled surfaces in every characteristic where Theorem \ref{kv} fails.

Nonetheless, assuming $B=0$, we have the following result (cf. \cite{mukai}).

\begin{thm}\label{char_pat}
Let $X$ be a smooth surface in positive characteristic. Assume that there exists a big and nef Cartier divisor $D$ on $X$ such that
$$H^1(X,\oo_X(K_X+D))\neq 0.$$
Then:
\begin{itemize}
\item $X$ is either quasi-elliptic of Kodaira dimension one or of general type;
\item up to a sequence of blow-ups, $X$ has the structure of a fibered surface over a smooth curve such that every fibre is connected and singular.
\end{itemize}
\end{thm}

Furthermore, Terakawa deduced the following vanishing result in \cite{terakawa}, using the techniques in \cite{shepherdbarron}.

\begin{thm}
Let $X$ be a smooth projective surface over an algebraically closed field of characteristic $p>0$ and let $D$ be a big and nef Cartier divisor on $X$.
Assume that either
\begin{enumerate}
\item $\kappa(X)\neq 2$ and $X$ is not quasi-elliptic with $\kappa(X)=1$; or
\item $X$ is of general type with
\begin{itemize}
\item $p\ge 3$ and $(D^2)>\vol(X);$ or
\item $p=2$ and $(D^2)>\max\{\vol(X), \vol(X)-3\chi(\oo_X)+2\}.$
\end{itemize}
\end{enumerate}
Then
$$H^i(X, \oo_X(K_X + D)) = 0$$
for all $i > 0$.
\end{thm}

Our aim is to improve this theorem for arbitrary surfaces, via bend-and-break techniques.

More generally, we want to deduce some results on the injectivity of cohomological maps
$$H^1(X,\oo_X(-D)) \xrightarrow{F^*} H^1(X,\oo_X(-pD))$$
where $D$ is a big divisor on $X$.

The following result by  Koll\'ar is an application of bend-and-break lemmas (cf. Theorem \ref{b-a-b}), specialised in our two-dimensional setting.

\begin{thm}\label{kollar}
Let $X$ be a smooth projective variety over a field of positive characteristic and let $D$ be a Cartier divisor on $X$ such that:
\begin{enumerate}
\item $H^1(X,\oo_X(-mD)) \xrightarrow{F^*} H^1(X,\oo_X(-pmD))$ is not injective for some integer $m>0$;
\item there exists a curve $C$ on $X$ such that
$$(p-1)(D\cdot C) - (K_X\cdot C) >0.$$
\end{enumerate}
Then through every point $x$ of $C$ there is a rational curve $C_x$ such that 
\begin{equation}\label{alg_rel}
[C]\overset{e}{\approx} k_0 [C_x] +\sum_{i\neq 0} k_i[C_i]
\end{equation} 
(as algebraic cycles), with $k_i\ge 0$ for all $i$ and
$$(p-1)(D\cdot C_x)-(K_X\cdot C_x)\le \dim(X)+1.$$
\end{thm}

\proof
This is essentially a slight modification of \cite[Theorem II.6.2]{kollar}. For the reader's convenience, we sketch it ab initio. Assumption (1) allows us to construct a finite morphism
$$\pi\colon Y \rightarrow X,$$
where $Y$ is defined as a Cartier divisor in the projectivisation of a non-split rank-two bundle over $X$ (cf. \cite[Construction II.6.1.6]{kollar}, which is a slight modification of Construction \ref{constr}). Furthermore, the following property holds:
$$K_Y=\pi^*(K_X+(k(1-p)D)),$$
where $k$ is the largest integer for which $H^1(X,-kD)\not=0$.

Now take the curve given in (2) and consider $C':=\red \pi^{-1}(C)$. The hypothesis on the intersection numbers and the formula for the canonical divisor of $Y$ guarantee that $(K_Y\cdot C')<0$. Let $y \in C'$ be a pre-image of $x$ in $Y$. So we can apply Theorem \ref{b-a-b} and deduce the existence of a rational curve $C'_y$ passing through $y$. Using the projection formula, we obtain a curve $C_x$ on $X$ for which:
$$(p-1)(D\cdot C_x)-(K_X\cdot C_x)\le \dim(X)+1.$$
\endproof

If we assume the dimension to be two and the divisor $D$ to be big and nef, the asymptotic condition
$$H^1(X,\oo_X(-mD))=0,$$
for $m$ sufficiently large is guaranteed by a result of Szpiro (cf. \cite{szpiro}).

This remark gives the following corollary.

\begin{corollary}\label{sz}
Let $X$ be a smooth projective surface over a field of positive characteristic and let $D$ be a big and nef Cartier divisor on $X$ such that $H^1(X,\oo_X(-D))\neq 0$. Assume there exists a curve $C$ on $X$ such that
$$(p-1)(D\cdot C) - (K_X\cdot C) >0.$$

Then through every point $x$ of $C$ there is a rational curve $C_x$ such that 
$$(p-1)(D\cdot C_x)-(K_X\cdot C_x)\le 3.$$
\end{corollary}

We will show later how Corollary \ref{sz} can be used to deduce an effective version of Kawamata-Viehweg-type vanishing for arbitrary smooth surfaces


In what follows, we will also need the following lemma on fibered surfaces, which explicitly gives a bound on the genus of the fibre with respect to the volume of the surface.

\begin{lemma}\label{volume_genus}
Let $f\colon X \rightarrow C$ be a fibered surface of general type and let $g$ be the arithmetic genus of the general fibre $F$. Then
$$\vol(X)\ge g-4.$$
\end{lemma}

\proof
We divide our analysis according to the genus $b$ of the base, after having assumed the fibration is relatively minimal (i.e. that $K_{X/C}$ is nef).
\begin{description}
\item [$b\ge 2$] In this case we can deduce a better estimate. Indeed, 
$$\vol(X)\ge (K_X^2) = (K_{X/C}^2)+8(g-1)(b-1) \ge 8g-8.$$
\item [$b=1$] In this case we need a more careful analysis, since in positive characteristic we cannot assume the semipositivity of $f_*K_{X/C}$. Nonetheless the following general formula holds:
\begin{equation}\label{rr}
\deg(f_*K_{X/C}) = \chi(\oo_X) - (g-1)(b-1),
\end{equation}
which specialises to
$$\deg(f_*K_{X/C}) = \chi(\oo_X)\ge 0.$$
Formula (\ref{rr}) can be obtained via Riemann-Roch, since we know that $R^1f_*K_{X/C}=\oo_C$ and that $R^1f_*nK_{X/C}=0$ for $n\ge 2$ by relative minimality.
The last inequality can be assumed by \cite[Theorem 8]{shepherdbarron1}. Furthermore, one can apply the following formula
$$\deg(f_*(nK_{X/C})) = \deg(f_*K_{X/C}) + \frac{n(n-1)}{2}(K_{X/C}^2).$$
Since $K_{X/C}$ is big, we deduce that
$$\deg(f_*(2K_{X/C}))\ge 1.$$
As a consequence, we can apply the results of \cite{atiyah} and deduce a decomposition of $f_*(2K_{X/C})$ into indecomposable vector bundles
$$f_*(2K_{X/C})= \bigoplus_i E_i,$$
where we can assume that $\deg(E_1)\ge 1$. This implies that all quotient bundles of $E_1$ have positive degree. We want to show now that there exists a degree-one divisor $L_1$ on $C$ such that $h^0(C,f_*(2K_{X/C})\otimes \oo_C(-L_1))\neq0$. 
\\But this is clear, since, for all degree-one divisor $L$ on $C$, one has that all quotient bundles of $f_*(2K_{X/C})\otimes \oo_C(-L)$ have degree zero and, up to a twisting by a degree-zero divisor on $C$, one can assume there exists a quotient
$$f_*(2K_{X/C}\otimes \oo_C(-L_1)) \rightarrow \oo_C \rightarrow 0.$$
This implies that $h^0(X,\oo_X(2K_{X/C}-F)) (= h^0(C,f_*(2K_{X/C})\otimes \oo_C(-L_1)))\neq0$, where $F$ is the general fibre of $f$ and, since $K_X = K_{X/C}$ is nef, that
$$(K_X \cdot (2K_X-F))\ge 0.$$
This gives the bound
$$\vol(X)\ge (K_X^2) \ge g-1.$$
\item[$b=0$] Also in this case we can assume that $\chi(\oo_X)\ge0$ and, as a consequence, that
$$\deg(f_*K_{X/\pp^1}) = \chi(\oo_X) + g-1 \ge g-1.$$
If $g\ge 6$,
$$\deg(f_*K_{X/\pp^1}) \ge 5.$$
This implies that $\deg(f_*K_X\otimes\oo_{\pp^1}(3)) \ge 0$ and, as a consequence of Grothendieck's theorem on vector bundles on $\pp^1$,
$$h^0(X,\oo_X(K_X-f^*\oo_C(-3)))\neq 0.$$
As before, we have assumed that $K_{X/\pp^1}= K_X + f^*\oo_C(2)$ is nef, so
$$((K_X + f^*\oo_C(2))\cdot (K_X-f^*\oo_C(-3)))\ge 0.$$
So, in this case
$$\vol(X)\ge (K_X^2) \ge 2g-2.$$
If $g\le 5,$ we simply use the trivial inequality $\vol(X)\ge 1$ to deduce
$$\vol(X) \ge g-4.$$
\end{description}
\endproof

Our result in this setting is an effective bound, depending only on the birational geometry of $X$, to guarantee the injectivity on the induced Frobenius map on the $H^1$'s.

\begin{thm}\label{injectivity}
Let $X$ be a smooth surface in characteristic $p>0$ and let $D$ be a big Cartier divisor $D$ on $X$.
Then, for all integers
$$m >m_0=\frac{2\vol(X)+9}{p-1},$$ 
the induced Frobenius map 
$$H^1(X,\oo_X(-mD)) \xrightarrow{F^*} H^1(X,\oo_X(-pmD))$$
is injective.
\\(If $\kappa(X)\neq 2$, the volume $\vol(X)=0$).
\end{thm}

\begin{rmk}
The previous result is trivial if $H^1(X,\oo_{X}(-D))=0$, Furthermore, combined with Corollary \ref{sz}, gives an effective version of Kawamata-Viehweg theorem (cf. Corollary \ref{finale}) in the case of big and nef divisors. Our hope is to generalise this strategy in order to deduce effective vanishing theorems also in higher dimension.
\end{rmk}

\proof
Assume, by contradiction, that
$$H^1(X,\oo_X(-\lceil m_0\rceil D)) \xrightarrow{F^*} H^1(X,\oo_X(-p\lceil m_0\rceil D))$$
has a nontrivial kernel. Then, after a sequence of blow-ups $f\colon X' \rightarrow X$, we can assume the existence of a (relatively minimal) fibration (possibly with singular general fibre) of arithmetic genus $g$
$$\pi\colon X' \rightarrow C.$$
We remark that we can reduce to prove our result on $X'$, since $D':=f^*D$ is a big divisor we have the following commutative diagram
$$\begin{xymatrix}
{
  H^1(X', \oo_{X'}(-\lceil m_0\rceil D')) \ar[r]^{F^*} \ar[d]_{\simeq} &  H^1(X', \oo_{X'}(-p\lceil m_0\rceil D') \ar[d]^{\simeq}\\
  H^1(X, \oo_X(-\lceil m_0\rceil D))  \ar[r]^{F^*}   &  H^1(X, \oo_X(-p\lceil m_0\rceil D))
 }
 \end{xymatrix},$$
where the vertical isomorphisms holds because of $R^1f_*\oo_{X'}=0.$
We can now apply Theorem \ref{kollar} to $\lceil m_0\rceil D'$: we can choose $C$ to be a general fibre $F$ of $\pi$, which certainly intersects positively $D',$ and we can use Lemma \ref{volume_genus} and obtain
\begin{equation}\label{>3}
(p-1)\lceil m_0\rceil (D'\cdot F) - (K_{X'}\cdot F)  \ge (p-1)\lceil m_0\rceil  -(2g-2) > 3.
\end{equation}
So we can apply Theorem \ref{kollar}: fix a point $x \in F$ and find a rational curve $C_x$ such that
$$(p-1)m_0(D\cdot C_x) - (K_X\cdot C_x) \le 3.$$
Notice that, by construction, $F=C_x$, because of (\ref{alg_rel}) in Theorem \ref{kollar}. But this is a contradiction, because of (\ref{>3}).
\endproof

We finally obtain our effective vanishing theorem.

\begin{corollary}\label{finale}
Let $X$ be a smooth surface in characteristic $p>0$ and let $D$ be a big and nef Cartier divisor $D$ on $X$.
Then, 
$$H^1(X, \oo_X(K_X+mD))=0$$
for all integers $m >m_0,$ 
where
\begin{itemize}
\item $m_0=\frac{3}{p-1}$ if $X$ is quasi-elliptic with $\kappa(X)=1$;
\item $m_0=\frac{2\vol(X)+9}{p-1}$ if $X$ is of general type.
\end{itemize}
\end{corollary}

\proof
For surfaces of general type, one simply applies the previous result. For quasi-elliptic surfaces, a better bound can be obtained, since, in this case, $(K_X\cdot F)=0$, where $F$ is the general fibre.

\endproof

\bibliographystyle{alpha}
\bibliography{./bibliografia}

\end{document}